\documentclass[11pt]{amsart}

\usepackage[T1]{fontenc}
\usepackage[utf8]{inputenc}
\usepackage{lmodern}
\usepackage{microtype}
\usepackage{mathtools}
\usepackage{amssymb}
\usepackage[hidelinks]{hyperref}

\hypersetup{%
  pdftitle={A Proof of the Imbalance Conjecture},
  pdfauthor={James Alexander Schreib; Yousof Yavari},
  pdfsubject={Graph theory; graphic degree sequences; edge imbalance},
  pdfkeywords={edge imbalance, graphic sequence, degree sequence, Erdos-Gallai theorem}
}

\numberwithin{equation}{section}

\newtheorem{theorem}{Theorem}[section]
\newtheorem{lemma}[theorem]{Lemma}

\DeclareMathOperator{\imb}{imb}
\newcommand{\mset}[1]{\left\{\!\left\{#1\right\}\!\right\}}

\title[A proof of the Imbalance Conjecture]{A Proof of the Imbalance Conjecture}

\author{James Alexander Schreib}
\address{New York University}
\email{jas10320@nyu.edu}

\author{Yousof Yavari}
\address{The University of British Columbia, Vancouver, British Columbia, Canada}
\email{yousofy@student.ubc.ca}

\subjclass[2020]{05C07, 05C99}
\keywords{edge imbalance, graphic sequence, degree sequence, Erd\H{o}s--Gallai theorem}
\date{August 15, 2026}

\begin{document}

\begin{abstract}
For an edge $uv$ of a finite simple graph $G$, its imbalance is
$|d_G(u)-d_G(v)|$, and the imbalance multiset $M_G$ consists of the
imbalances of all edges of $G$.  Kozerenko and Skochko conjectured that
$M_G$ is graphic whenever every edge has positive imbalance.  We prove
this conjecture.  The main ingredient is the following capacity bound:
for every set $A$ of $k$ edges,
\[
  \sum_{e\in E(G)\setminus A}\min\{k,\imb_G(e)\}
  \ge k\max\{\Delta-k,0\},
\]
where $\Delta$ is the maximum degree of $G$.  This bound yields all
Erd\H{o}s--Gallai inequalities directly; a parity computation completes the
proof.
\end{abstract}

\maketitle

\section{Introduction}

All graphs in this paper are finite, simple, and undirected.  For a vertex
$v$ of a graph $G$, write $d_G(v)$ for its degree.  The \emph{imbalance} of
an edge $uv\in E(G)$ is
\[
  \imb_G(uv)=|d_G(u)-d_G(v)|.
\]
The sum of all edge imbalances was introduced by Albertson as a measure of
graph irregularity~\cite{Albertson1997}.  Kozerenko and Skochko subsequently
studied the multiset of the individual edge imbalances and formulated the
following conjecture~\cite{KozerenkoSkochko2014}.  Kozerenko and Serdiuk later
recalled the conjecture and reported that it had been computationally verified
for all graphs of order at most twelve~\cite{KozerenkoSerdiuk2023}.
After the initial version of this manuscript was posted, we became aware of an
earlier proof by A.~A. Raoui, publicly posted on Zenodo on June 8,
2026~\cite{Raoui2026}.  The present proof was obtained independently.

A finite multiset of nonnegative integers is called \emph{graphic} if it is
the multiset of vertex degrees of a finite simple graph.  We write
\[
  M_G=\mset{\imb_G(e):e\in E(G)}
\]
for the imbalance multiset of $G$.

\begin{theorem}[Imbalance Conjecture]\label{thm:main}
Let $G$ be a finite simple graph.  If
\[
  \imb_G(e)>0 \qquad\text{for every }e\in E(G),
\]
then $M_G$ is graphic.
\end{theorem}

\section{Preliminaries}

We use the following classical criterion for graphic sequences.

\begin{theorem}[Erd\H{o}s--Gallai~\cite{ErdosGallai1960}]\label{thm:EG}
Let $n\ge 1$ and let
\[
  x_1\ge x_2\ge\cdots\ge x_n\ge0
\]
be integers.  Then $(x_1,\ldots,x_n)$ is the degree sequence of a finite
simple graph if and only if $\sum_{i=1}^n x_i$ is even and, for every
$1\le k\le n$,
\begin{equation}\label{eq:EG}
  \sum_{i=1}^k x_i
  \le k(k-1)+\sum_{i=k+1}^n\min\{k,x_i\}.
\end{equation}
\end{theorem}

For an integer $q$, write $q^{+}=\max\{q,0\}$.

\section{A capacity bound}

\begin{lemma}[Capacity bound]\label{lem:capacity}
Let $G$ satisfy the hypothesis of Theorem~\ref{thm:main}, and suppose that
$E(G)\ne\varnothing$.  Let
\[
  \Delta=\max_{v\in V(G)}d_G(v).
\]
For every $A\subseteq E(G)$ with $|A|=k$,
\begin{equation}\label{eq:capacity}
  \sum_{e\in E(G)\setminus A}\min\{k,\imb_G(e)\}
  \ge k(\Delta-k)^{+}.
\end{equation}
\end{lemma}

\begin{proof}
The assertion is immediate when $k=0$ or $k\ge\Delta$.  Suppose that
$1\le k<\Delta$, and let $r=\Delta-k$.  Choose a vertex $x$ with
$d_G(x)=\Delta$.  Let $a$ be the number of edges of $A$ incident with
$x$.  Exactly $\Delta-a$ edges incident with $x$ lie outside $A$.  Since
$a\le k$,
\[
  \Delta-a\ge\Delta-k=r.
\]
Choose $r$ of these edges, and let $R$ be the set of their endpoints other
than $x$.  Let $X$ be the set of endpoints other than $x$ of the remaining
$(\Delta-a)-r$ edges incident with $x$ outside $A$.  Since $G$ is simple,
all these endpoints are distinct, and
\[
  |R|=r,
  \qquad
  |X|=(\Delta-a)-r=k-a.
\]

For every $z\in R$, maximality of $d_G(x)$ and the hypothesis
$\imb_G(xz)>0$ imply
\[
  d_G(z)<\Delta
  \qquad\text{and}\qquad
  \imb_G(xz)=\Delta-d_G(z).
\]
Consequently,
\begin{equation}\label{eq:reservoir}
  \min\{k,\imb_G(xz)\}+(d_G(z)-r)^{+}=k.
\end{equation}
Indeed, if $d_G(z)\le r$, the two terms on the left are $k$ and $0$;
if $d_G(z)>r$, they are $\Delta-d_G(z)$ and $d_G(z)-r$.

Set
\[
  \rho=\sum_{z\in R}(d_G(z)-r)^{+}.
\]
Call an edge an \emph{escape edge} if it has one endpoint in $R$ and the
other outside $R\cup\{x\}$, and let $F$ be the set of escape edges.  For
each $z\in R$, at most $r$ of its neighbors lie in $R\cup\{x\}$: there
are at most $r-1$ possible neighbors in $R$ and the additional vertex
$x$.  Thus $z$ is incident with at least $(d_G(z)-r)^{+}$ escape edges.
Every escape edge has exactly one endpoint in $R$, so summing over $R$
counts each edge of $F$ exactly once.  Writing $N_G(z)$ for the neighborhood
of $z$, we obtain
\begin{align}
  |F|
  &=\sum_{z\in R}\bigl|N_G(z)\setminus(R\cup\{x\})\bigr|\notag\\
  &\ge\sum_{z\in R}(d_G(z)-r)^{+}\notag\\
  &=\rho.\label{eq:escape}
\end{align}

Summing \eqref{eq:reservoir} over $z\in R$ gives
\begin{align*}
  \sum_{z\in R}\min\{k,\imb_G(xz)\}
  &=\sum_{z\in R}\bigl(k-(d_G(z)-r)^{+}\bigr)\\
  &=rk-\rho.
\end{align*}
Thus the $r$ selected edges from $x$ to $R$ contribute exactly $rk-\rho$
to the left-hand side of \eqref{eq:capacity}.

No escape edge is incident with $x$, while exactly $a$ edges of $A$ are
incident with $x$.  Thus at most $k-a$ escape edges belong to $A$.  By
\eqref{eq:escape},
\begin{align*}
  |F\setminus A|
  &=|F|-|F\cap A|\\
  &\ge\rho-(k-a).
\end{align*}
Since $|F\setminus A|\ge0$, it follows that
\begin{equation}\label{eq:escape-outside}
  |F\setminus A|\ge(\rho-(k-a))^{+}.
\end{equation}
Each edge in $F\setminus A$ contributes at least one to the left-hand side
of \eqref{eq:capacity}, because every edge has positive imbalance.  The
$k-a$ edges from $x$ to $X$ also lie outside $A$ and each contributes at
least one.  The selected edges from $x$ to $R$, the edges from $x$ to $X$,
and the edges in $F\setminus A$ are pairwise disjoint.  Therefore
\begin{align*}
  \sum_{e\in E(G)\setminus A}\min\{k,\imb_G(e)\}
  &\ge rk-\rho+(k-a)+(\rho-(k-a))^{+}\\
  &=rk+(k-a-\rho)^{+}\\
  &\ge rk\\
  &=k(\Delta-k).
\end{align*}
Here the equality uses $q+(-q)^{+}=q^{+}$ with
$q=k-a-\rho$, and the following inequality uses $q^{+}\ge0$.  This proves
\eqref{eq:capacity}.
\end{proof}

\section{Proof of the conjecture}

\begin{proof}[Proof of Theorem~\ref{thm:main}]
If $E(G)=\varnothing$, then $M_G$ is the degree multiset of the empty graph.
Hence assume that $E(G)\ne\varnothing$, and let
\[
  m=|E(G)|,
  \qquad
  \Delta=\max_{v\in V(G)}d_G(v).
\]

The endpoints of every edge have distinct positive degrees, both at most
$\Delta$.  Therefore
\begin{equation}\label{eq:imb-upper}
  \imb_G(e)\le\Delta-1
  \qquad\text{for every }e\in E(G).
\end{equation}

Let $A\subseteq E(G)$ and write $|A|=k$.  By
\eqref{eq:imb-upper} and Lemma~\ref{lem:capacity},
\begin{align}
  \sum_{e\in A}\imb_G(e)
  &\le k(\Delta-1)\notag\\
  &\le k(k-1)+k(\Delta-k)^{+}\notag\\
  &\le k(k-1)
     +\sum_{e\in E(G)\setminus A}\min\{k,\imb_G(e)\}.
  \label{eq:subset-EG}
\end{align}
For the middle inequality, equality holds when $k\le\Delta$, while for
$k\ge\Delta$ it follows from $\Delta-1\le k-1$.

List the elements of $M_G$ in nonincreasing order:
\[
  x_1\ge x_2\ge\cdots\ge x_m>0.
\]
For each $1\le k\le m$, choose $A\subseteq E(G)$ to consist of $k$ edges
whose imbalances are $x_1,\ldots,x_k$.  Then \eqref{eq:subset-EG} becomes
\[
  \sum_{i=1}^k x_i
  \le k(k-1)+\sum_{i=k+1}^m\min\{k,x_i\}.
\]
Thus all the Erd\H{o}s--Gallai inequalities hold.

It remains to verify parity.  For all integers $p$ and $q$,
\[
  |p-q|\equiv p+q\pmod2.
\]
Consequently,
\begin{align*}
  \sum_{e\in E(G)}\imb_G(e)
  &\equiv \sum_{uv\in E(G)}\bigl(d_G(u)+d_G(v)\bigr) \pmod2\\
  &=\sum_{v\in V(G)}d_G(v)^2\\
  &\equiv\sum_{v\in V(G)}d_G(v)\pmod2\\
  &=2m\equiv0\pmod2.
\end{align*}
Hence $\sum_{i=1}^m x_i$ is even.  By Theorem~\ref{thm:EG},
$(x_1,\ldots,x_m)$ is the degree sequence of a finite simple graph $H$.
Thus $\mset{d_H(z):z\in V(H)}=M_G$, so $M_G$ is graphic.
\end{proof}

\clearpage
\section*{Acknowledgements}

The authors thank Eric Hou
(\href{mailto:yhou15@student.ubc.ca}{\texttt{yhou15@student.ubc.ca}})
for carefully checking and verifying Yousof Yavari's initial proof and for
helping improve both its argument and exposition.

\section*{Generative-AI disclosure}

Each author independently discovered a proof of the conjecture.  James
Alexander Schreib discovered his proof on July 24, 2026, using OpenAI's
GPT-5.6 Sol Pro, and Yousof Yavari discovered his proof on August 1, 2026,
using OpenAI's GPT-5.6 Sol Max.  The proof presented in this manuscript is the
merged and polished version of the two independently discovered proofs.
Both authors subsequently reviewed and edited the merged proof and take
responsibility for every definition, calculation, logical inference, and
conclusion in the manuscript.  A machine-checked Lean 4 formalization of the
proof is available at
\url{https://github.com/jamesschreib/imbalance-conjecture}.

\end{document}